\documentclass[a4paper]{amsart}
\usepackage{tikz}
\usepackage{nicefrac}
\usepackage{tkz-berge}
\usepackage{tikz-cd}
\usetikzlibrary{quotes, graphs, positioning}
\usetikzlibrary{graphs.standard}
\usepackage{mathtools,amssymb,hyperref}

\newcommand{\TC}{{T_\CC}}
\newcommand{\TR}{T}
\newcommand{\GR}{G_\RR}
\newcommand{\GC}{G_\CC}
\newcommand{\RR}{\mathbb{R}}
\newcommand{\CC}{\mathbb{C}}

\newcommand{\ZZ}{\mathbb{Z}}

\newcommand{\PP}{\mathbb{P}}
\newcommand{\Xss}{X^{\text{ss}}}
\newcommand{\Xps}{X^{\text{ps}}}
\DeclareMathOperator{\pt}{pt}
\DeclareMathOperator{\Div}{Div}

\let\SO\relax
\DeclareMathOperator{\SO}{SO}
\DeclareMathOperator{\diag}{diag}
\DeclareMathOperator{\conv}{conv}

\DeclareMathOperator{\id}{id}

\newtheorem{theorem}{Theorem}[section]

\newtheorem{lemma}[theorem]{Lemma}
\newtheorem{proposition}[theorem]{Proposition}

\newtheorem{corollary}[theorem]{Corollary}

{\theoremstyle{remark}
\newtheorem{remark}[theorem]{Remark}

}

\theoremstyle{definition}

\newtheorem{example}[theorem]{Example}
\newtheorem{definition}[theorem]{Definition}

\title[Orbit spaces of oriented Grassmannians of planes]{Orbit spaces of maximal torus actions on oriented Grassmannians of planes}
\author[H. S\"u{\ss}]{Hendrik S\"u\ss}
\address{Hendrik S\"u\ss\\ School of Mathematics,
The University of Manchester,
Alan Turing Building,
Oxford Road
Manchester M13 9PL}
\email{\href{mailto:hendrik.suess@manchester.ac.uk}{hendrik.suess@manchester.ac.uk}}

\begin{document}
\maketitle
\begin{abstract}
Motivated by Buchstaber's and Terzi\'c' work on the complex Grassmannians $\GC(2,4)$ and $\GC(2,5)$ we describe the moment map and the orbit space of the oriented Grassmannians $\GR^+(2,n)$ under the action of a maximal compact torus. Our main tool is the realisation of these oriented Grassmannians as smooth complex quadric hypersurfaces and the relatively simple Geometric Invariant Theory of the corresponding algebraic torus action. 
\end{abstract}

\section{Introduction}
We denote an algebraic torus $(\CC^*)^k$ by $\TC$ and the corresponding compact torus $(S^1)^k \subset \TC$ by $\TR$. A complex algebraic variety with a $\TC$-action is called a $\TC$-variety. The \emph{complexity} of a $\TC$-variety is the minimal (complex) codimension of an orbit. In this paper we study the $\TR$-orbit spaces of projective $\TC$-varieties and apply our findings to the case of oriented Grassmannians of planes and that of smooth $\TC$-varieties of complexity $1$. Our main goal is to determine the corresponding $\TR$-orbit spaces up to homeomorphism.

We consider the Grassmannian $\GR^+(2,n)$ parametrising oriented planes in $\RR^2$ with the natural action of a maximal torus in $\SO_n$. Our main result determines the orbit space of this action.

\begin{theorem}
\label{thm:grassmannian}
  The orbit space $\GR^+(2,n)/\TR$ is homeomorphic to the join 
\[S^{\lfloor n/2 \rfloor-1} \ast \PP_\CC^{\lceil n/2 \rceil-2}.\]
\end{theorem}

For smooth varieties with a torus action of complexity $1$ we derive the following general results on the structure of their orbit spaces.

\begin{theorem}
\label{thm:cplx-1}
Consider a smooth projective $\TC$-variety $X$ of complexity $1$. Then the corresponding orbit space $X/\TR$ is a topological manifold with boundary.
\end{theorem}

\begin{theorem}
\label{thm:cplx-1-sphere}
Consider a smooth projective $\TC$-variety $X$ of complexity $1$ with only finitely many lower dimensional $\TC$-orbits. Then the orbit space $X/\TR$ is homeomorphic to a sphere.
\end{theorem}

Note, that results comparable to Theorem~\ref{thm:cplx-1} have been proved by Ayzenberg in \cite{ayz18} and by Vladislav Cherepanov in his unpublished diploma thesis. Moreover, Theorem~\ref{thm:cplx-1-sphere} was proved independently, but using similar methods, by Karshon and Tolman in  \cite{kt-top}. Their work covers the more general setting of symplectic manifolds with Hamiltonian torus actions. In their paper they also prove Theorem~\ref{thm:grassmannian} for the cases of complexity $1$, i.e. for $n=5,6$.

\medskip
Our main tool is Geometric Invariant Theory (GIT) and its symplectic counterpart in combination with the Kempf-Ness Theorem. This approach suggest to stratify the manifold and eventually the orbit space via a polyhedral subdivision of the momentum polytope, which encodes the variation of GIT quotients. In general these stratifications can become arbitrarily  complicated. However, in the cases considered in this paper they turn out to be almost trivial allowing us to derive concrete results about the orbits spaces.

In Section~\ref{sec:t-varieties} we fix our setting for compact torus actions induced by algebraic torus actions on complex varieties and recall crucial results from Geometric Invariant Theory. Moreover, we derive first results on the structure of orbits spaces in suitable situations. We then apply these to the special cases of oriented Grassmannians of planes in Section~\ref{sec:orient-grassm} and $\TC$-varieties of complexity $1$ in Section~\ref{sec:complexity-one}.

In order to distinguish between the algebraic and the topological category, we are going to denote isomorphism of algebraic varieties by $\cong$ and homeomorphisms of topological spaces by $\approx$.

\subsection*{Acknowledgement}
This research was supported by the program 
\emph{Interdisciplinary Research} conducted jointly by the Skolkovo Institute of Science and Technology and the Interdisciplinary Scientific Center J.-V. Poncelet. In particular, I am grateful for the warm hospitality offered by Center Poncelet. This work was also partially supported by the grant 346300 for IMPAN from the Simons Foundation and the matching 2015-2019 Polish MNiSW fund. I would like to thank Anton Ayzenberg, Victor Buchstaber, Alexander Perepechko and Nigel Ray for stimulating discussions on the subject of this paper.
Finally, I want to thank Yael Karshon  and Susan Tolman for making their draft \cite{kt-top} available to me.

\section{$\TC$-varieties and their $\TR$-orbit spaces}
\label{sec:t-varieties}
Fix a linearised action of an algebraic torus $\TC=(\CC^*)^k$ on $\PP^N_\CC$ with weights $u_0, \ldots, u_N \in \ZZ^k$, i.e. for $t=(t_1,\ldots,t_k) \in \TC$ 
we have
\[t.(z_0:\ldots:z_N) = (t^{u_0}z_1: \ldots: t^{u_N}z_N),\]
where $t^{u_j} := t_1^{(u_j)_1}\cdots t_k^{(u_j)_k}$.
Then a moment map of this action is given by
\[
\nu:\PP^N_\CC \to \RR^k; \quad (z_0: \ldots z_N) \mapsto \frac{\sum_j |z_j|^2 u_j}{\sum_j |z_j|^2}.
\]
For an embedded projective variety $X \subset \PP^N_\CC$, which is invariant under under this torus action, a moment map of the induced torus action on $X$ is given by the restriction $\mu = \nu|_X$. The moment image $P = \mu(X)$ is known to be a convex polytope \cite{zbMATH03757356,zbMATH03793067}. We start with some notions known as variation of GIT quotients with \cite{zbMATH00108945,zbMATH01560390,zbMATH01308132} being the most relevant references. For a point $x \in X$ the moment image $\Delta(x)=\mu(\overline{\TC.x}) \subset P$ of its orbit closure is again a polytope and the orbit $\TC.x$ is mapped to the relative interior $\Delta^{\!\circ}(x) \subset \Delta(x)$.

For a point $u \in P$ we define
\[\Xss(u)=\{x \in X \mid u \in \Delta(x)\}, \quad \Xps(u)=\{x \in X \mid u \in \Delta^\circ(x)\}.\]
Hence, $\Xss(u)$ consists of those points in $X$ whose orbit closures intersect $\mu^{-1}(u)$ and $\Xps(u)$ consists of those points whose orbits intersect $\mu^{-1}(u)$. Equivalenty $\Xps(u)$ is the union of closed $\TC$-orbits in $\Xss(u)$. 

Now for every $u\in P$ we may consider 
\begin{equation}
\lambda(u)=\bigcap_{x, u \in \Delta(x)}\Delta(x); \quad \lambda^\circ(u)= \bigcap_{x, u \in \Delta^{\!\circ}(x)}\Delta^\circ(x) \label{eq:stratification1}
\end{equation}
Since only finitely many polytopes occur as moment images of orbit closures their intersections are again polytopes. We denote the set of all these polytopes $\lambda(u)$ by $\Lambda$. This set is partially ordered by the face relation $\prec$. The polytopes $\lambda \in \Lambda$ form a polyhedral subdivision of $P$ and one obtains a stratification of $P$ via their relative interiors.
\begin{equation}
P = \bigsqcup_{\lambda \in \Lambda} \lambda^\circ.\label{eq:stratification2}
\end{equation}
For $u\in P$ let us denote by $\lambda(u)$ the unique element of $\lambda \in \Lambda$ such that $u \in \lambda^\circ$.

From the definitions above it follows that $\Xss(u)=\Xss(v)$ if and only if 
$\lambda(u)=\lambda(v)$, i.e. $u$ and $v$ are contained in the relative interior of the same element of $\Lambda$. In this case also $\Xps(u)=\Xps(v)$ holds. Hence, we may define $\Xss_\lambda = \Xss(u)$ and $\Xps_\lambda = \Xps(u)$  for $u \in \lambda^\circ$.

\begin{example}
  \label{exp:p2}
We consider the linear $T=\CC^*$-action on $\PP_\CC^2$ given by $t.(x:y:z) = (tx:t^{-1}y:z)$. Then  the moment map is given by $\mu(x:y:z)= \frac{|x|^2-|y|^2}{|x|^2+|y|^2+|z|^2}$. We get $P=\mu(\PP^2)=[-1,1] \subset \RR$. The orbits can be described as follows. We have the fixed points  $(1:0:0)$, $(0:1:0)$ and $(0:0:1)$ with moment images $1$, $-1$ and $0$, respectively. The moment images of the other $\TC$-orbits are
\begin{align*}
\mu(T.(1:0:1))&=(0,1) \subset \RR,\\
\mu(T.(0:1:1))&=(-1,0) \subset \RR,\\
\mu(T.(1:1:0)) &= (-1,1) \subset \RR,\\
\mu(T.(\alpha:1:1))&=(-1,1) \subset \RR,\, \alpha \in \CC^*.
\end{align*}
Hence, in this case $\Lambda$ is obtained by subdividing the interval $[-1,1]$ at the point $0$, or more formally \(\Lambda=\{[-1,0],[0,1],\{-1\},\{0\},\{1\}\}.\)
\end{example}

\medskip
By the Kempf-Ness Theorem for rational values of $u$ the definition of $\Xss(u)$ coincides with the semi-stable locus of Mumford's Geometric Invariant Theory.
Hence, there exists a categorical quotient morphism
$q_\lambda \colon \Xss_\lambda \to Y_\lambda = \Xss_\lambda/\!\!/\TC$ where
$Y_\lambda$ is an orbit space for the $\TC$-action on $\Xps_\lambda$ and
the corresponding quotient map is given by the restriction of $q_\lambda$ to $\Xps_\lambda$. The occurring quotients $Y_u$ have the expected dimension for $u\in P^\circ$, but can be lower-dimensional for elements $u \in \partial P$. By \cite[Lem.~7.2]{zbMATH03880866} for $u \in \lambda^\circ$ every $\TC$-orbit in $\Xps_\lambda$ intersects $\mu^{-1}(u)$ in exactly one $\TR$-orbit. Hence, the restriction of $q_\lambda$ to $\mu^{-1}(u)$ induces a homeomorphisms between $Y_\lambda = \Xss_\lambda/\!\!/\TC$ and the topological orbit space $\mu^{-1}(u)/\TR$.  Moreover the inclusion $X^\lambda \subset X^\gamma$ for $\gamma \prec \lambda$ induces contraction morphisms on the level of quotients $p_{\gamma\lambda} \colon Y_\lambda \to Y_\gamma$ forming an inverse system.

\begin{example}
  \label{exp:toric}
  When $\dim X = \dim \TC$, i.e. if the variety is toric, then the moment image $P$ completely determines the variety. The moment images $\Delta(x)$ of $\TC$-orbit closures are just the faces of the polytope $P$ and the stratification of $P$ is the decomposition of $P$ into the relative interiors of its faces. The preimage $\mu^{-1}(u)$ consists of exactly on $\TR$-orbit with dimension equal to dimension of the face containing $u$ in its interior. Consequently $X/\TR \approx P$ with $\mu$ coinciding with the quotient map. Alternatively, we may apply
  Proposition~\ref{prop:trivial-git} below and obtain $X/\TR \approx S^{k-1}*\{\pt\} \approx D^k \approx P$.
\end{example}

\begin{example}
\label{exp:toric-downgrade}
Consider a projective toric variety $X$ corresponding to a polytope $Q \subset \RR^d$. Then the inclusion of a $k$-dimensional subtorus $\TC' \subset \TC$ induces a surjection $F:\RR^d \to \RR^k$. Given a moment map $\mu$ for the $\TC$-action a corresponding moment map $\mu'\colon X \to \RR^k$ is given by $\mu'=F \circ \mu$. Hence, the moment image for the $\TC'$-action is $P:=F(Q)$ and the stratification of $P$ is induced by the images of the faces of $Q$. More precisely,
the stratification consists of the relative interiors of the polytopes
\[\lambda(u) = \bigcap_{\tau \prec Q,\;u \in F(\tau)}\!\! F(\tau).\]

Moreover, the GIT quotients $\Xss(u)/\!\!/\TC'$ are again toric varieties corresponding to the polytope $F^{-1}(u) \cap Q$, see \cite[Prop.~3.5]{zbMATH00034160}. 
\end{example}

Already in \cite{zbMATH04029746}  it has been observed that orbit space of the $\TR$-action on $X$ can be constructed out of the inverse system of GIT quotients.
\begin{theorem}[{\cite[\S5]{zbMATH04029746}}]
\label{thm:goresky-macpherson}
  We have
\[X/\TR \approx \left(\bigsqcup_{\lambda \in \Lambda} \lambda \times Y_\lambda\right)/{\sim},\]
where $(u,y) \sim (u,y')$ if $(u,y) \in \gamma \times Y_\gamma$,
$(u,y') \in \lambda \times Y_\lambda$ with $\gamma \prec \lambda$ 
and $p_{\gamma\lambda}(y')=y$.
\end{theorem}

We easily derive the following result, which turns out to be a little bit handier in some situations.

\begin{corollary}
\label{cor:orbit-space}
 Assume that we have a compact topological space $Y$ and with proper surjective maps $r_\lambda \colon Y \to Y_\lambda$ being compatible with the inverse system above. Then we have the following homeomorphism.
\[X/\TR \approx (P \times Y)/{\sim_r}.\]
Here, the equivalence relation is generated by 
\[(u,y) \sim_r (u,y') \Leftrightarrow r_{\lambda(u)}(y) = r_{\lambda(u)}(y')\]
 for $u \in \lambda^\circ$.
\end{corollary}
\begin{proof}
  There is a canonical map 
\[P \times Y \to \left(\bigsqcup_{\lambda \in \Lambda} \lambda \times Y_\lambda\right)/{\sim};\quad (u,y) \mapsto [(u,r_{\lambda(u)}(y)]\]
This map is surjective and continuous and indentifies exactly those pairs which are equivalent under $\sim_r$. The quotient $(P \times Y)/{\sim_r}$ is compact as $(P \times Y)$ is and by Theorem~\ref{thm:goresky-macpherson} the codomain of the map is homeomorphic to $X/T$, which is a Hausdorff space. Hence, the induced continuous bijection $(P \times Y)/{\sim_r} \to X/T$ is a homeomorphism.
\end{proof}


\begin{remark}
  In \cite{2018arXiv180206449B,2018arXiv180305766B} such a $Y$ is called a \emph{universal parameter space} for the $\TC$-obits.  In algebraic geometry a natural choice for such a dominating algebraic object $Y$ would be the inverse limit of the $Y_\lambda$ or the Chow quotient of $X$ by $\TC$, which can be identified with a distinguished irreducible component of this inverse limit.
\end{remark}

If the structure of the inverse system $\{Y_\lambda\}_{\lambda \in \Lambda}$ of GIT quotients is complicated Corollary~\ref{cor:orbit-space} might not give much concrete information about the orbit space $X/\TR$. However, in certain situations this structure turns out to be almost trivial allowing us to effectively calculate the orbit space.

\begin{definition}
  We say the $\TC$-action on $X \subset \PP^N_\CC$ has an \emph{almost trivial variation of GIT} if for $\lambda \not \subset \partial P$ the quotients $Y_\lambda$ are all isomorphic to some $Y$.
\end{definition}

\begin{example}
  If the torus action has complexity one than the quotients $Y_\lambda$ are smooth algebraic curves or just a point, where the latter happens at most over the boundary of $P$. The only contraction morphisms here are isomorphisms or the contraction of a curve to a point. Hence, the definition is automatically fulfilled.
\end{example}

\begin{proposition}
\label{prop:trivial-git}
  Consider a $\TC$-action on $X$ with almost trivial variation of GIT and only finitely many lower-dimensional $\TC$-orbits. Then $X/\TR$ is homeomorphic to the topological join $S^{k-1} \ast Y$.
\end{proposition}
\begin{proof}
  Having an almost trivial variation of GIT means that $Y_\lambda \cong Y$ for $\lambda \not \subset \partial P$. On the other hand, having only finitely many lower-dimensional $\TC$-orbits implies that the moment fibre of a boundary point $u \in \partial P$ consists of exactly one (lower-dimensional) orbit and therefore $Y_{\lambda(u)}$ is just a point. Hence, by Corollary~\ref{cor:orbit-space} we have
\[
X/\TR \approx (P \times Y)/{\sim_\partial},
\]
where the equivalence relation $\sim_{\partial}$ is generated by $(u,y) \sim_{\partial} (u,y')$ for $u \in \partial P$. Now, the claim follows from Lemma~\ref{lem:join} below.
\end{proof}

\begin{lemma}
\label{lem:join}
  Consider the closed unit disc $D^k$ and the unit sphere $S^{k-1}$. Then for any compact topological manifold $Y$  we have
\[S^{k-1} \ast Y \approx (D^k \times Y)/{\sim_\partial},\]
where the equivalence relation $\sim_{\partial}$ is generated by $(u,y) \sim_{\partial} (u,y')$ for $u \in \partial D^k$.
\end{lemma}
\begin{proof}
  Recall that the join 
$S^{k-1} \ast Y$ is defined as $(S^{k-1} \times Y \times [0,1])/{\sim}$, with the equivalence relation being generated by $(s,y,0) \sim (s',y,0)$ and $(s,y,1) \sim (s,y',1)$. Now the homeomorphism is given by 
\[ (S^{k-1} \times Y \times [0,1])/{\sim} \longrightarrow (D^{k} \times Y)/{\sim_\partial} ,\quad [(u,y,t)] \mapsto [(tu,y)].\]
\end{proof}

As a special case of Lemma~\ref{lem:join} we may consider the situation when $Y \approx S^m$. Then Lemma~\ref{lem:join} implies $(D^k \times S^m)/{\sim_\partial} \approx S^{k-1} \ast S^m$, which is known to be homeomorphic to $S^{m+k}$. The lemma below gives a slightly more general statement.

\begin{lemma}
\label{lem:holed-spheres}
  For any closed $H \subset \RR^k$ we have
  \[((D^k \cap H) \times S^m)/{\sim_\partial} \;\approx\; S^{k+m} \cap (H\times \RR^m)\]
  where the equivalence relation  $\sim_{\partial}$ is generated by  $(u,y) \sim_{\partial} (u,y')$ for $u \in \partial D^k$.
\end{lemma}
\begin{proof}
  We can state the homeomorphism explicitely
  \[
  ((D^k\cap H) \times S^m)/{\sim_\partial} \to S^{k+m} \cap H \times \RR^m, \quad [(u,y)] \mapsto
  (u,\sqrt{1-|u|^2}\cdot y).
  \]
  For every $(u,y) \in (\partial D^k \cap H) \times S^m$ we have $(u,\sqrt{1-|u|^2}\cdot y)=(u,0)$. Hence, the map is a well-defined continuous bijection from a compact space to a Hausdorff space and, therefore a homeomorphism.
\end{proof}

\section{Oriented Grassmanians of planes as $\TC$-varieties}
\label{sec:orient-grassm}
We consider the smooth manifold $\GR^+(2,n)$ parametrising oriented planes in $\RR^n$. An oriented plane is given by an orthonormal basis $(v_1,v_2)$. Another orthonormal pair $(v_1',v_2')$ gives rise to the same oriented plane if and only if
$(v_1',v_2') = (v_1,v_2)Q_\phi$  with
\[
Q_\phi = \begin{pmatrix}
  \cos \phi & -\sin \phi\\
  \sin \phi & \cos \phi
\end{pmatrix} \in \SO_2.
\]
Hence,
\[\GR^+(2,n) = \{(v_1,v_2) \in \RR^{n\times 2} \mid \langle v_i,v_j \rangle = \delta_{ij}\}/\SO_2.\]
A (compact) torus action on $\GR^+(2,n)$ is induced by the choice of a maximal torus in $\SO_n$ via its action on the pair $(v_1,v_2)$. A maximal torus $\TR$ is given by block diagonal matrices of the form
$\diag(Q_{\phi_1}, \ldots, Q_{\phi_k})$ in the case $n=2k$ or $\diag(Q_{\phi_1}, \ldots, Q_{\phi_k},1)$ in the case $n=2k+1$. In the even-dimensional case the induced action on the oriented planes is not effective as $-I$ acts trivially. To obtain an effective action one has to pass to the quotient $\TR/\langle \pm I \rangle$. However, this does not effect the orbit structure of the action.

It is well-known that the oriented Grassmannian of planes can be identified with the underlying smooth manifold of the complex smooth quadric $Q^{n-2}$ in $\PP^{n-1}_\CC$, see e.g. \cite[p. 280]{zbMATH03280665}. Indeed, the map
\begin{align*}
  \Psi \colon \RR^{n\times 2} \to \CC^n; \quad (v_1,v_2) \mapsto w = v_1 + i\cdot v_2
\end{align*}
 induces an embedding $\bar \Psi \colon \GR^+(2,n) \hookrightarrow \PP^n_\CC$.
This is well-defined as $\Psi((v_1,v_2)Q_\phi) = e^{i\phi}\cdot \Psi(v_1,v_2)$. Moreover the condition
$\langle v_1 , v_2 \rangle = 0$ is equivalent to $\Im(\sum_j w_j^2)=0$ and
$|v_1|^2/|v_2|^2 = 1$ is equivalent to $\Re(\sum_j w_j^2)=0$. Hence, the image of the embedding in $\PP^n_\CC$ is cut out by the equation $\sum_j w_j^2 = 0$. A change of coordinates 
\begin{align*}
  z_{2j-1} = w_{2j-1}+ i\cdot w_{2j}; \quad z_{2j} = w_{2j-1}-i\cdot w_{2j} \quad \text{ for } j=1,\ldots,k
\end{align*}
in the case $n=2k$ and additionally $z_n=w_{n}$ in the case $n=2k+1$ leads to the equivalent equation
\begin{equation}
\sum_{j=1}^k z_{2j-1}z_{2j} = 0 \label{eq:even}
\end{equation}
or 
\begin{equation}
z_n^2 + \sum_{j=1}^k z_{2j-1}z_{2j}=0,\label{eq:odd}
\end{equation}
respectively. Now in these coordinates one easily checks that for an oriented plane $E \in \GR^+(2,n)$ with
\[\bar \Psi(E) = (z_1: \ldots :z_n)\]
 we have
\[\bar \Psi(\diag(Q_{\phi_1}, \ldots, Q_{\phi_k})E)
= \left(e^{i\phi_1} z_1: e^{-i\phi_1} z_2:\ldots : e^{i\phi_k} z_{2k-1}: e^{-i\phi_1} z_{2k}\right)
\]
in the case $n=2k$ and similarly
\[\bar \Psi(\diag(Q_{\phi_1}, \ldots, Q_{\phi_k},1)E)
= \left(e^{i\phi_1} z_1: e^{-i\phi_1} z_2:\ldots : e^{i\phi_k} z_{2k-1}: e^{-i\phi_1} z_{2k}:z_{2k+1}\right)
\]
in the case $n=2k+1$. Let $e_j$ denote the $j$th canonical basis vector of $\ZZ^k$. Then the action of $\TR=(S^1)^k$ above is induced by an algebraic torus action of $\TC=(\CC^*)^k$ with weights 
\[\deg(z_{2j-1})=e_j, \; \deg(z_{2j})=-e_j; \quad j=1,\ldots,k.\]
and $\deg(z_{2k+1})=0$ in the case $n=2k+1$.

We are now going to describe the GIT quotients in order to eventually construct the orbit space using Corollary~\ref{cor:orbit-space}. We only describe the case of even $n=2k$ in detail. The situation for $n$ odd is very similar.
 
The moment map is given by
\begin{equation}
\mu(z_1: \ldots : z_n)=\frac{1}{\sum_{i=1}^n |z_i|^2}\sum_{j=1}^{k} (|z_{2j-1}|^2 - |z_{2j}|^2)e_j.\label{eq:momentmap}
\end{equation}
The moment image of $X$ is the cross-polytope given as the convex hull of the weights $P=\beta_k =\conv(\pm e_1, \ldots, \pm e_k)$. The fixed point $(1:0:\ldots:0)$ is mapped to $e_1$, $(0:1:\ldots:0)$ to $-e_1$ and similarly for the other coordinates.

\begin{remark}
\label{rem:cross-polytope-boundary}
  The proper faces of the cross-polytope $P$ are exactly the convex hulls of subsets of $\{\pm e_1, \ldots, \pm e_k\}$ where for every $j$ at most one of $e_j$ and $-e_j$ is contained.
\end{remark}



\begin{lemma}
  \label{lem:boundary-preimages}
    The moment preimage of a boundary point $u \in \partial P$
    consists of exactly one $\TR$-orbit. Hence, the quotient $Y_{\lambda(u)}$ is just a single point.
\end{lemma}
\begin{proof}
  It follows from Remark~\ref{rem:cross-polytope-boundary} that a point $(z_1: \ldots: z_n) \in X$ is mapped to the boundary of $P$ if and only if all the products $z_{2j-1}z_{2j}$ vanish. Indeed, assume $\mu(z)$ lies in the convex hull $\conv(\sigma_1e_1, \ldots,\sigma_ne_n)$, where $\sigma_i \in \{-1,1\}$ for $i=1,\ldots,n$. Then the coefficients of $\mu(z)$ in the corresponding barycentric coordinates are up to sign the same coefficients as in (\ref{eq:momentmap}). Hence, we must have that
  \begin{equation*}
    \frac{\sum_{j=1}^{k} \left||z_{2j-1}|^2 - |z_{2j}|^2\right|}{\sum_{i=1}^n |z_i|^2} = 1
  \end{equation*}
  or equivalently
    \begin{equation*}
    \sum_{j=1}^{k} \left||z_{2j-1}|^2 - |z_{2j}|^2\right| = \sum_{j=1}^{k} \left(|z_{2j-1}|^2 + |z_{2j}|^2\right).
  \end{equation*}
  This implies that for every $j=1,\ldots,k$ either $z_{2j-1}=0$ or $z_{2j}=0$ .

  Now, assume $z=(z_1:\ldots: z_n)$ and  $z'=(z_1':\ldots: z_n')$ have the same moment image and that  $z_{2j-1}z_{2j} = z'_{2j-1}z_{2j}' = 0$, for $j=1,\ldots k$. By choosing a suitable representative of the homogeneous coordinates we may assume that $\sum_j |z_j|^2 =\sum_j |z_j'|^2 = 1$. 
  With these choice of homogenous coordinates $|z_{2j-1}|^2-|z_{2j}|^2 = |z_{2j-1}|^2-|z_{2j}|^2$ holds for $j=1,\ldots k$, since $\mu(z)=\mu(z')$. For sign reasons we have either $z_{2j-1}=z_{2j-1}'=0$ or $z_{2j}=z_{2j}'=0$. This implies
   $|z_{2j}|=|z_{2j}'|$ or  $|z_{2j-1}|=|z_{2j-1}'|$, respectively. In either case we have $(z_{2j-1}',z_{2j}') = (s_j \cdot z_{2j-1},s^{-1}_j \cdot z_{2j})$ for some element $s_j \in S^1 \subset \CC^*$. Hence, $z$ and $z'$ lie in the same $\TR$-orbit.
\end{proof}

We consider the rational map 
\[q \colon \PP^{n-1}_\CC \dashrightarrow \PP^{k-2}_\CC, \quad (z_1:\ldots:z_n) \mapsto (z_3z_4:\ldots:z_{n-1}z_n).\]
This map is easily seen to be invariant under the $\TC$-action. It is well-defined on the locus of points where at least one of the products $z_{2j-1}z_{2j}$ for $j=2,\ldots,k$ does not vanish. For a point $z \in X$ this is equivalent to the fact that $\mu(z) \in P^\circ$.

\begin{lemma}
\label{lem:q-quotient-map}
  For $u \in P^\circ$ the map $q|_{\mu^{-1}(u)} \colon \mu^{-1}(u) \to \PP^{k-2}_\CC$ is a quotient map to the $\TR$-orbit space of the fibre.
\end{lemma}
\begin{proof}
  Consider $z,z' \in X$ with $\mu(z)=\mu(z') \in P^{\circ}$ and $q(z)=q(z')$.
  By choosing a suitable representative of the homogeneous coordinates for $z$ and $z'$ we may assume that $z_{2j-1}z_{2j}=z'_{2j-1}z'_{2j}$ for $j=2,\ldots,k-1$.  Then the defining equation of $X\subset \PP^{n-1}_\CC$ implies $\sum_j z_{2j-1}z_{2j}=\sum_j z'_{2j-1}z'_{2j}=0$. Hence, also $z_1z_2=z_1'z_2'$ must hold.
  Let us set $N=\sum_i |z_i|^2$ and  $N'=\sum_i |z_i'|^2$.
  Assume $z_{2j-1}z_{2j} = 0$ then $\mu(z)=\mu(z')$ implies
  \[\frac{|z_{2j-1}|^2-|z_{2j}|^2}{N} = \frac{|z_{2j-1}|^2-|z_{2j}|^2}{N'}.\]
  Hence, for sign reasons we have $z_{2j-1}=z'_{2j-1}=0$ or $z_{2j}=z'_{2j}=0$ in this case. Now, for each $j=1, \ldots k$ we set
  $s_j=z'_{2j-1}/z_{2j-1}$ or $s_j=z_{2j}/z'_{2j}$ whichever is defined.
  If they are both defined they have to coincide, since $z_{2j-1}z_{2j}=z'_{2j-1}z'_{2j}$. If $z_{2j-1}=z'_{2j-1}=z_{2j}=z'_{2j}=0$, then we set $s_j=1$. By these choices we have $s.z=z'$ with $s=(s_1,\ldots,s_k) \in \TC$. It remains to show, that $s \in \TR \subset \TC$.

  W.l.o.g we may assume that 
  \begin{equation}
    |s_1| = \max \{|s_1|, |s_1^{-1}| ,\ldots, |s_k|, |s_k^{-1}|\}.
    \label{eq:max}
\end{equation}
  The condition $\mu(z)=\mu(z')$  implies
   \[\frac{|z_1|-|z_2|}{N} = \frac{|z_1'|-|z_2'|}{N'},\]
   Note, that by (\ref{eq:max}) we have $|s_1|^{-2} \leq \frac{N'}{N} \leq |s_1|^2$. Now, from $(z'_1,z_2')=(s_1z_1,s_1^{-1}z_2)$ we obtain
  \[\frac{|z_1|^2-|z_2|^2}{N} = \frac{|s_1|^2|z_1|^2-|s_1|^{-2}|z_2|^2}{N'},\]
  which implies \(({N'}/{N}-|s_1|^2)|z_1|^2 = ({N'}/{N} - |s_1|^{-2})|z_2|^2.\)
  For sign reasons this is only possible if $|s_1|^2=N'/N=1$.
  Now, it follows from  (\ref{eq:max}) that $|s_1|= \ldots = |s_k|=1$ and $s \in \TR$.
\end{proof}

Since $\Xss(u)$ consists exactly of the orbits whose closures intersect $\mu^{-1}(u)$ it follows also that $q|_{\Xss(u)}$ is a good quotient in the sense of Geometric Invariant Theory. Hence, it coincides with the GIT quotient.

\begin{proof}[Proof of Theorem~\ref{thm:grassmannian}]
  We use Corollary~\ref{cor:orbit-space}. Here, $Y$ is just given as $Y=\PP^{\lceil n/2\rceil-2}$ and by Lemma \ref{lem:boundary-preimages} we have $r_{\lambda(u)}\colon Y \to \{\pt\}$ for $u \in \partial P$ and by \ref{lem:q-quotient-map} $r_{\lambda(u)} =\id_{Y}$. In particular, the equivalent relation on $P \times Y$ is just given by $(u,y) \sim_{r} (u,y')$ for $u \in \partial P$. Now, the claim follows from Lemma~\ref{lem:join}.
\end{proof}

We conclude this section by studying the moment images $\Delta(x)$ of $\TC$-orbit closures and the induced subdivision of $P$ from Section~\ref{sec:t-varieties}. For $n=2k+1$ the convex hull of every subset of vertices of $P$ occurs as a moment image. For $n=2k$ such convex hulls are moment images if an only if they are faces of $P$ or contain at least two pairs of opposite vertices $\{e_i,-e_i\}$, $\{e_j,-e_j\}$.
Indeed, for every such polytope $\Delta$ a corresponding $\TC$-orbit is given by
$\TC\cdot(z_1:z_2: \ldots :z_n)$ with $z_{2j-1} \neq 0 \Leftrightarrow e_j \in \Delta$ and $z_{2j} \neq 0 \Leftrightarrow -e_i \in \Delta$ for $j=1, \ldots, k$. Note, that for $n=2k+1$ such $(z_1:z_2: \ldots: z_n)$ fulfilling (\ref{eq:odd}) always exist, but for $n=2k$ there is obviously no non-trivial solution of  (\ref{eq:even}) where all but one monomial vanish.
In both cases, with the exception of $k=2$, the induced subdivision of $P$ is the same and coincides with the stellar subdivision of $P$ obtained by starring in the origin.

\begin{remark}
In \cite{2018arXiv180305766B} Buchstaber and Terzi\'c introduced the notion of $(2n,k)$-manifold. It's relatively straightforward to check that for $n=2k$, the axioms of this notion are indeed fulfilled for the associated effective torus action by $\TC/\langle \pm 1\rangle$. However, in the odd case on a generic point of the hyperplane section $[z_n=0]$ we have finite stabilisers of order $2$, which violates the conditions for a  $(2n,k)$-manifold.
\end{remark}

\begin{remark}
  Note that $\GR^+(2,6)$ can be identified with $\GC(2,4)$ as both are given by the smooth quadric hypersurface in $\PP^5_\CC$. Hence, for this case we just rediscover the results of \cite{zbMATH06783273}. Combinatorially this fact is reflected by  coincidence of the moment polytopes, i.e. the cross-polytope $\beta_3$ and the hypersimplex $\Delta_{4,2}$.
\end{remark}
\section{Complexity-one $\TC$-varieties}
\label{sec:complexity-one}
If the complexity of the torus action is $1$ the possible GIT quotients $Y_\lambda$ are either single points or isomorphic to a fixed algebraic curve $Y$. Hence, when applying Corollary~\ref{cor:orbit-space} to this situation the maps $p_\lambda \colon Y \to Y_\lambda$ are either isomorphisms or contractions to a point.
Our main aim in this section is to prove that in this situation the resulting orbits spaces are topological manifolds with boundary and even spheres if the number of lower dimensional $\TC$-orbits is finite.

\begin{remark}
  In the toric case the orbit space can be identified with the moment polytope. In particular, it is also a topological manifold with boundary. Hence, Theorem~\ref{thm:cplx-1} can be seen as generalisation of this fact. On the other hand, this phenomenon is very special to complexity $0$ and $1$. In higher dimensions this will almost never be the case. For example for a smooth projective variety $Y$ the join $S^n \ast Y$, which occurs as an orbit space in the situation of Proposition~\ref{prop:trivial-git}, is a topological manifold if and only if $Y \cong \PP^1_\CC$.
\end{remark}

\begin{proposition}
  \label{prop:projective-space-cplx-1}
  Consider the projective $d$-space with a $(d-1)$-torus $\TC$ acting effectively by weights $u_0=0, u_1, \ldots, u_d \in \ZZ^{d-1}$ on the coordinates $z_0,\ldots, z_d$. Then the orbit space $\PP^d_\CC/\TR$ is homeormorphic to either a disc or a sphere. In particular, it is a topological manifold with boundary.
\end{proposition}

\begin{proof}
The moment image of $\PP^d_\CC$ is given by the convex hull of the weights $u_0, \ldots, u_d$. The weights $u_0, \ldots, u_d$ are necessarily affinely dependent in $\RR^{d-1}$. On the other hand they span $\RR^{d-1}$ as an affine space due to the effectiveness of the torus action.  Hence, there is a non-trivial choice of $\alpha_j \in \ZZ$, such that \(0=\sum_{i=0}^d \alpha_i u_i\) and $0=\sum_{i=0}^d \alpha_i$ and the coefficients are unique up to simultaneous scaling. 

Set $K=\{i \in \{0,\ldots,d\} \mid \alpha_i \neq 0\}$. Then $P$ is obtained as the join $Q * \Delta$ of the lower-dimensional polytopes $Q = \conv\{u_i\}_{i\in K}$ and $\Delta=\conv\{u_i\}_{i \notin K}$ of dimensions $m:=(\#K-2)$ and $n:=(d -\#K)$, respectively. Here, we allow that $\Delta = \emptyset$ and use the non-standard convention  $Q * \emptyset := Q$.  Note, that $\Delta$ is a simplex (or empty). Hence, a $u \in \Delta$ has a unique representation as $u = \sum_{j \in K} \lambda_j u_j$ with $\lambda_j \geq 0 $ and $\sum_j \lambda_j =1$. For $u \in Q$  such a representation $u = \sum_{j \notin K} \lambda_j u_j$ is unique if and only if $u \in \partial Q$. It follows that $u \in P=Q*\Delta$ has a unique such representation if and only if $u \in \partial Q * \Delta$.

Now, $Y_{\lambda(u)} = \mu^{-1}(u)/\TR$  is a point whenever $u \in P$ has a unique representation as $u = \sum_j \lambda_j u_j$ and $\mu^{-1}(u)/\TR \approx \PP^1_\CC$ otherwise. This is just a special case of Example~\ref{exp:toric-downgrade}, when $F:\RR^d \to \RR^k$ is given by $F(e_i)=u_i$ for $i=1, \ldots, d$. Then the intersection of $F^{-1}(u)$ and the standard simplex consists of all linear combinations $\sum_{i=1}^n \lambda_i e_i$ with non-negative coefficients, such that $u = \sum_{i=1}^n \lambda_i u_i$ and $\sum \lambda_i=1$. The result is a point if the linear combination is unique or a line segment if not. The corresponding toric varieties are a single point and $\PP^1_\CC$, respectively. Alternatively, it not hard to show that the non-trivial quotient morphisms $\Xss_\lambda \to Y_\lambda = \PP^1$ are all restrictions of the rational map \[\PP^d \dashrightarrow \PP^1, \qquad (z_0: \ldots : z_d) \mapsto \left(z_0^{\sum_i \alpha_i} : \prod_{i} z_i^{\alpha_i}\right).\]


From Corollary~\ref{cor:orbit-space} we obtain that 
\[\PP^d_\CC/\TR \approx \big(P\times \PP^1_\CC\big)/\sim,\]
where the equivalence relation $\sim$ is generated by $(u,y) \sim (u,y')$ with $y,y' \in \PP^1_\CC$ and $u \in \partial Q * \Delta \subset \partial P$. Topologically $\partial Q$ can be identified with a sphere $S^{m-1}$ and $Q$ with the cone $S^{m-1}*\{\pt\}$. Similarly we have a homeomorphism $\Delta \approx S^{n-1}*\{\pt\}$ for $\Delta \neq \emptyset$.
For the pair $(P,\partial Q * \Delta)$ we obtain 
\begin{align}
  (P,\partial Q * \Delta) &\approx (S^{m-1}*\pt*S^{n-1}*\{\pt\},\; S^{m-1}*(S^{n-1}*\{\pt\})) \\
  &\approx  (S^{m+n-1}*\{\pt\}*\{\pt\},\; S^{m+n-1}*\{\pt\}) \\
  &\approx (D^{m+n}*\{\pt\},\;  D^{m+n}).
\end{align}
Note, that the disc $D^{m+n}$ can be identified with the hemisphere via projection and $D^{m+n}*\{\pt\}$ with the corresponding halfdisc. Now, by choosing $H$ to be an arbitrary halfspace it follows from Lemma~\ref{lem:holed-spheres} that the orbit space is a hemisphere.

For $\Delta=\emptyset$ and $P=Q$  we see directly $(P,\partial Q * \Delta) = (Q,\partial Q) \approx (D^{d-1},S^{d-2})$ and we obtain $\PP^d_\CC/\TR \approx S^{d+1}$ from invoking Lemma~\ref{lem:holed-spheres} again, this time with $H=\RR^{d-1}$.
\end{proof}

\begin{proof}[Proof of Theorem~\ref{thm:cplx-1}]
  We first consider the situation of a complexity-one torus action on the affine space $\CC^d$. Such an action is linearisable by \cite{zbMATH03262169}.  We may equivariantly compactify the $\TC$-action on $\CC^d$ to a $\TC$-action on $\PP^d_\CC$. Then $\CC^d/\TR$ is an open subset of $\PP^d_\CC/\TR$. Hence, the claim follows from the observation in Proposition~\ref{prop:projective-space-cplx-1} that the orbit space $\PP^d_\CC/\TR$ is a manifold with boundary.

To deduce the general case we consider the two situations from Lemma~\ref{lem:2-cases-cplx-1}. If contractions to a point do not occur the equivalence relation $\sim_r$ is trivial and by Corollary~\ref{cor:orbit-space} we have $X/\TR \approx P\times Y$ which is a product of topological manifolds with boundary.

In the second situation, we have $Y \cong \PP^1$. Then by \cite[Thm~5]{zbMATH06322161} we have an equivariant open cover of $X$ by copies of $\CC^n$ and the result follows directly from the consideration above.
\end{proof}

\begin{remark}
  To deduce the general case from the case $X = \CC^d$ in the proof of Theorem~\ref{thm:cplx-1} we may alternatively consider the induced $\TC$-action on the tangent space at a fixed point and reduce everything to this situation by applying Luna's Slice Theorem.
\end{remark}

 \begin{lemma}[{\cite[Lemma~5.7]{zbMATH01663168}}]
   \label{lem:2-cases-cplx-1}
  In the situation of a complexity-one $\TC$-action on a smooth projective variety either
  \begin{enumerate}
  \item  the GIT quotients are all isomorphic to $Y$, or
  \item $Y \cong \PP^1_\CC$.
  \end{enumerate}
\end{lemma}

Consider the following definition.

\begin{definition}
  A $k$-holed $n$-sphere is defined as the intersection of $S^n \subset \RR^{n+1}$ with $k$ affine closed halfspaces, such that $S^n$ is contained in none of them and their boundary hyperplanes intersect only outside the sphere. 
\end{definition}

In other words, we remove the interiors of $k$ disjoint closed discs from $S^n$.

\begin{proposition}
\label{prop:combinatorial-holed-sphere}
  Consider a smooth $\TC$-variety $X$ of complexity $1$ with moment polytope $P$ and stratification $\Lambda$. Let $\Lambda'$ be the set of $\lambda \in \Lambda$, such that $\lambda \subset \partial P$ and $Y_\lambda \neq \{\pt\}$. Assume that $\Lambda'$ consists of $k$ disjoint polytopes. Then $X/\TR$ is a $k$-holed sphere.
\end{proposition}
\begin{proof}
  In our situation the equivalence relation $\sim_r$ from Corollary~\ref{cor:orbit-space} is generated by $(u,y) \sim_r (u,y')$ for $u \in  \partial P \setminus \bigcup_{\lambda \in \Lambda'}\lambda^\circ$. Note that by the preconditions $\partial P \setminus \bigcup_{\lambda \in \Lambda'}\lambda^\circ$ is a $k$-holed sphere. Hence, by Lemma~\ref{lem:holed-spheres} the orbit space $X/\TC \approx (P \times \PP^1)/\sim_r$ is a $k$-holed sphere as well.
\end{proof}

This statement is useful to determine the orbit space in concrete situation. To demonstrate this we look at the classification of Fano threefolds from \cite{Mori1981}. In \cite{suess:picbook,autfano} the ones within the classification which admit a $\TC$-action of complexity $1$ where identified. In the following we determine the orbit spaces for all of them.


\begin{theorem}
  Using the notation of the Mori-Mukai classification \cite{Mori1981} for the Fano threefolds with $(\CC^*)^2$-action we obtain the orbits spaces being  $k$-holed spheres with $k=0$ for $Q$, 2.24, 2.29, 2.32, 3.10 and 3.20 with $k=1$ for 2.30, 2.31, 3.18, 3.21, 3.23 and 3.24, with $k=2$ for 3.19, 3.22, 4.4, 4.5, 4.7 and 4.8. 
\end{theorem}
\begin{proof}
  For every of the above Fano varieties the combinatorial data provided in \cite[Sec. 5]{suess:picbook} consists of the moment polytope $P$ and a piecewise linear map
\[\Psi \colon P \to \Div_\RR \PP^1_\CC\]
from the polytope to the vector space of $\RR$-divisors on $\PP^1$. By
\cite[Sec.~3.1]{suess:picbook} we have $Y_{\lambda(u)}=\{\pt\}$ if and only $\deg \Psi(u)=0$. Now, one checks that in every case the subset $\{u \in \partial P \mid \deg \Psi(u)>0\} \subset \partial P$ consists of the interior of $k$ disjoint facts of $P$. Applying Proposition~\ref{prop:combinatorial-holed-sphere} gives the desired result.
\end{proof}

\begin{remark}
  Note, that $Q$ is the smooth quadric and by Section~\ref{sec:orient-grassm} coincides with $\GR^+(2,5)$. Moreover, 2.32 is the variety of complete flags in $\CC^3$. Hence, we recover results of \cite[Prop. 8]{2018arXiv180305766B} and \cite{kt-top}.
\end{remark}

\begin{proof}[Proof of Theorem~\ref{thm:cplx-1-sphere}]
  To have finitely many lower dimensional $\TC$-orbits implies that the GIT quotients $Y_u$ for $u \in \partial P$ are just points. Then by Lemma~\ref{lem:2-cases-cplx-1} we conclude that $Y_u = \PP_\CC^1$ for every $u \in P^\circ$. By applying Proposition~\ref{prop:trivial-git} we obtain 
\[X/\TR \approx \PP^1 \ast S^{k-1} \approx S^2 \ast S^{k-1} \approx S^{k+2}.\]
\end{proof}

\enlargethispage{0.7cm}
\bibliography{topquot}
\bibliographystyle{halpha}
\end{document}